\documentclass[12pt]{article}
\usepackage{graphicx}
\usepackage{subcaption}

\usepackage{amsmath,amsthm,amssymb}

\newtheorem{thm}{Theorem}[section]

\newtheorem{prop}[thm]{Proposition}

\newtheorem{remark}[thm]{Remark}

\newcommand{\R}{\mathbb{R}}

\newcommand{\Z}{\mathbb{Z}}

\newcommand{\bW}{\mathbf{W}}
\newcommand{\bB}{\mathbf{B}}
\newcommand{\bD}{\mathbf{D}}
\newcommand{\be}{\mathbf{e}}
\newcommand{\ba}{\mathbf{a}}
\newcommand{\bc}{\mathbf{c}}
\newcommand{\bd}{\mathbf{d}}

\newcommand{\by}{\mathbf{y}}

\newcommand{\bxe}{\mathbf{xe}}

\newcommand{\ie}{i.e.,\ }

\newcommand\bal{{\boldsymbol\alpha}}

\newcommand{\hby}{\hat \by}
\newcommand{\hba}{\hat \ba}
\newcommand{\hy}{\hat y}
\newcommand{\ha}{\hat a}

 \newcommand{\E}{{\mathbb E}}

\usepackage{bigints}
\usepackage{relsize}

\begin{document}

\title{
Reproduction Capabilities of Penalized Hyperbolic-polynomial Splines}

\author{Rosanna Campagna$^{+}$, Costanza Conti$^{*}$\\
\small{$^+$Dept. of Mathematics and Physics, University of Campania ``L. Vanvitelli'', Italy} \\
 \small{$^{*}$Department of Industrial Engineering, DIEF, University of  Firenze, Italy}\\
  \small{ \texttt{rosanna.campagna@unicampania.it;  costanza.conti@unifi.it} }
}

 \maketitle

\begin{abstract}
This paper investigates two important analytical properties of hyperbolic - polynomial penalized splines, HP-splines for short. HP-splines, obtained by combining a special type of difference penalty with hyperbolic-polynomial B-splines (HB-splines), were recently introduced by the authors as a generalization of P-splines. HB-splines  are bell-shaped basis functions consisting of segments made of real exponentials $e^{\alpha x},\, e^{-\alpha x}$ and linear functions multiplied by these exponentials, $xe^{+\alpha x}$ and $xe^{-\alpha x}$. Here, we show that these type of penalized splines reproduce 
function in the space $\{e^{-\alpha x},\ x e^{-\alpha x}\}$, that is  they  fit exponential data exactly.
Moreover, we show  that they conserve  the first and second \lq exponential\rq\ moments. 
\end{abstract}


\section{Introduction}\label{sec:1}
This short paper investigates the reproduction capabilities of  hyperbolic-polynomial penalized splines.  HP-splines, were  recently introduced in \cite{CC2021} as a generalization of  the better known P-splines (see \cite{EilersMarx96, EMD2015}), and combine  a finite difference penalty with HB-splines that  piecewise consist of real exponentials and monomials multiplied by these exponentials. 
Numerical examples show that the exponential nature of HP-splines may turn out to be useful in applications when the data show an exponential trend  \cite{CampagnaContiCuomo2020}.

The HP-splines  we consider have segments in  the four-dimensional space
\begin{equation}\label{def:E4}
\E_{4,\alpha}:=\textrm{span}\{e^{\alpha x},\ x\, e^{\alpha x},\ e^{-\alpha x},\  x\, e^{-\alpha x}\},\ \alpha \in  \R,
\end{equation}
with the  frequency $\alpha$ being an extra parameter to tune the smoother effects. Even though all details concerning their definition and construction can be already find in \cite{CC2021}, the analysis of their reproduction capability is there missing. To fill the gap, here we show that these type of penalized splines reproduce 
function in the space $\{e^{-\alpha x},\ x e^{-\alpha x}\}$, that is fit exponential data of the latter type exactly.
We also show  that they conserve  the first and second \lq exponential\rq\ moments, showing that HP-splines are the natural generalization of P-splines even with respect to reproduction and moment preservation.

\smallskip  Given the data points $(x_i,y_i),\ i=1,\ldots, m$, $x_1<\cdots <x_m$, the uniform knot partition $\Xi: = \{x_1:=a=\xi_1<\xi_2 \cdots  <\xi_n=b=:x_m\}$ with knots distance $h$, and denoting by $\{B^\alpha_0,\cdots, B^\alpha_{n+1}\}$ a HB-splines basis of the spline space with segments in $\E_{4,\alpha}$, 
 the HP-spline approximating the given data is obtained by solving the minimization problem
\begin{equation}\label{def:problemEP}
\min_{a_0,\ldots, a_{n+1}}
\sum_{i=1}^m w_i\left(y_i-\sum_{j=0}^{n+1}a_j B^\alpha_j(x_i)\right)^2  
+ \lambda  \sum_{j=2}^{n+1}\left((\Delta_2^{h,\alpha} \ba)_j \right)^2,
\end{equation}
where  the minimum is   with respect to the HB-splines coefficients $\ba=(a_j)_{j=0}^{n+1}$. The values $(w_1,\ldots, w_m)$ are non-zero weights,  $\Delta^{h,\alpha}_2$ is the difference operator acting on functions and on sequences respectively as (see \cite{CONTI2022126796} for this type of operators), 
$$
\Delta^{h,\alpha}_2\,u=u(x)-2e^{-\alpha h}u(x-h)+e^{-2\alpha h}u(x-2h),\quad \hbox{$h>0$ the knots distance}
$$
$$
(\Delta^{h,\alpha}_2\,\ba)_j=a_j-2e^{-\alpha h}a_{j-1}+e^{-2\alpha h}a_{j-2},\quad j\in \Z,$$    
and $\lambda$ is a 
regularization parameter that can be set in several different ways,
e.g. with the discrepancy principle, the generalized cross-validation,
or the L-curve method (see \cite{Hansen_book}, for example) \\
It is not difficult to see that the HP-spline can be written as 
\begin{equation}\label{def:hp}
s_{hp}(x)=\sum_{j=0}^{n+1}\ha_j B^\alpha_j(x),
\end{equation}
where $\hba=(\ha_0,\ldots, \ha_{n+1})^T\in  {\mathbb R}^{(n+2)}$
 is the solution of the linear system
\begin{equation}\label{expansion2}
\left(
{\bB^\alpha}^T{\mathbf W}
{\bB^\alpha}+ \lambda  
{\bD^{h, \alpha}_2}^T
{\bD^{h, \alpha}_2}\right)\ba={\bB^\alpha}^T{\mathbf W}{\mathbf y},
\end{equation} 
where ${\bf y}=(y_1,\ldots, y_{m})^T\in  {\mathbb R}^{m}$, 
${\bB^\alpha}\in {\mathbb R}^{m\times (n+2)}$ is the banded  collocation matrix ${\bB}^{\alpha}:=(B^{\alpha}_j(x_i))_{i=1,\ldots,m}^{j=0,\ldots,n+1}$,
${\bW}$ is the diagonal matrix $${\bW}:=(diag(w_i))_{i=1,\ldots,m} \in {\mathbb R}^{m\times m}$$ and  
\[{\bD}_2^{h, \alpha}= \left[ \begin{array}{cccccc}
1       & -2e^{-\alpha \,h} & e^{-2\alpha \,h}        & 0       & \cdots      & 0 \\
  0        & 1       & -2e^{-\alpha \,h} & e^{-2\alpha\, h}        & \cdots      & 0 \\
\vdots     & \ddots             & \ddots     & \ddots  &  \cdots                  & \vdots\\
\vdots     & \vdots     & \vdots      & 1       & -2e^{-\alpha\, h} & e^{-2\alpha\, h}       \\
\end{array} \right]\in{\mathbb R}^{n\times (n+2)}.\]
Note that for $\alpha\rightarrow 0$ the space $\E_{4,\alpha}$ reduces to $\{1, x, x^2, x^3\}$,  HB-splines reduce to cubic B-splines and the difference operator $\Delta^{h,\alpha}_2$ reduces to the standard forward second order difference operator acting on a sequence $\ba$ as $(\Delta_2\,\ba)_j=a_j-2a_{j-1}+a_{j-2}$. Therefore, for $\alpha=0$ HP-splines coincide with P-splines based on classical cubic B-splines proposed by Eilers and Max (see their recent monograph \cite{eilers_marx_2021}).  From now on, without loss of generality, we continue by assuming that ${\bW}$ is the identity matrix. 

\smallskip P-splines are known to have a number of useful properties, essentially inherited from B-splines and from the special type of penalty: they can fit polynomial data exactly, they can conserve the first  two moments of the data and show no boundary effects. The aim of this paper is to investigate similar reproduction properties of HP-splines.  As in the polynomial case, the HP-spline properties are essentially inherited from HB-splines and this is why  in Section \ref{sec:2}  we first prove that HB-splines reproduce $\E_{4,\alpha}$. Then, in Section \ref{sec:3}, we show that, whatever the value of the smoothing parameter $\lambda$, HP-splines fit exponential data exactly as they reproduces 
$\E_{2,-\alpha}:=\{e^{-\alpha x},\ x e^{-\alpha x}\}$. Moreover, we show that HP-splines conserve  the first and second \lq exponential\rq\ moments. Section \ref{sec:conclusion} draws conclusion and  highlights future works.
\section{HB-splines and their reproduction properties} \label{sec:2}
  As shown in \cite{Unser}, in the \lq cardinal\rq\, situation --corresponding to integer spline knots--, HB-splines can be defined through convolution. For the spline space with segment in $\E_{4,\alpha}$, 
 starting with the first order B-spline
\begin{equation}\label{def:B1}
B^1_{\alpha}(x)=e^{\alpha\,x}\chi_{[0,1]}(x),
\end{equation}
supported in $[0,1]$, the four order cardinal HB-spline supported on $[0,4]$ is obtained as
$$
B^1_{\bal}=\left( B^1_{\alpha}\ast  B^1_{\alpha}\ast B^1_{-\alpha}  \ast B^1_{-\alpha} \right)\quad\hbox{with}\quad \bal=(\alpha,\alpha,-\alpha, -\alpha),
$$
hence,  for any integer $k$, the HB-spline supported in $[k, k+4]$ is obtained by translation as $B^1_{\bal}(\cdot-k)$.
In case the knots are uniform but with a distance $h\neq 1$,   the corresponding HB-splines are defined by dilation, 
\begin{equation}\label{def:dilatata}
B^h_{\bal}=B^1_{\bal}(\frac \cdot h)=\left( B^1_{\alpha}\ast  B^1_{\alpha}\ast B^1_{-\alpha}  \ast B^1_{-\alpha} \right)(\frac \cdot h),
\end{equation}
and then translation. 
Alternatively, we directly start with the scaled order-one HB-spline $B^h_{\alpha}(x)=e^{\frac{\alpha}{h} x }\chi_{[0,h]}(x)$ and use repeated  convolution. In that case we  see that when dealing with grid spacing $h$, the frequency $\alpha$ is scaled into $\frac{\alpha}{h}$, a fact that will also enter into the exponential reproduction discussion we are going to make.

\smallskip Concerning the  HB-spline reproduction, we prove the following result.
\begin{prop}\label{prop:HBreproduction}
Let $\bal=(\alpha,\alpha,-\alpha, -\alpha)$ be the vector of frequencies,  and $B^h_{\bal h}$ the HB-spline function with uniform knots defined in \eqref{def:dilatata} having support $[0,4h]$.
Then, 
$$
f=\sum_{k\in \Z} c^f_k B^h_{\bal h}(\cdot-hk),\quad \hbox{for all}\quad f\in \E_{4,\alpha}.
$$
\end{prop}
\begin{proof}The starting point is \cite[Proposition 2]{Unser} that yields the particularly simple reproduction formulas for the order-two HB-spline $B^1_{\alpha,\alpha}=B^1_{\alpha}\ast B^1_{\alpha}$ 
$$
\sum_{k\in \Z} e^{\alpha\, k} B^1_{\alpha,\alpha}(x-k)=e^{\alpha\, x} , \quad \sum_{k\in \Z} (k+1)e^{\alpha\, k} B^1_{\alpha,\alpha}(x-k)=xe^{\alpha\, k},
$$
providing, for $t_k=kh$, the reproduction formula for $e^{\alpha\, x}$. In fact, from $e^{\frac{\alpha}{h}\, x}
=\displaystyle{\sum_{k\in \Z} e^{\frac{\alpha}{h}\, t_k} B^h_{\alpha,\alpha}(x-t_k)}$ we arrive at  
\begin{equation}\label{def:E2gen0}
e^{\alpha\, x}=\sum_{k\in \Z} e^{\alpha\, t_k} B^h_{\alpha h,\alpha h }(x-t_k).
\end{equation}
Similarly, $\sum_{k\in \Z} (k+1)e^{\alpha\, k} B_{\alpha,\alpha}(x-k)=xe^{\alpha\, k}$ gives $$\frac{x}{h}e^{\frac{\alpha}{h}\, x}=\displaystyle{\sum_{k\in \Z} (k+1)e^{\frac{\alpha}{h}\, kh} B^h_{\alpha,\alpha}(x-hk)}$$ and hence 
the reproduction formula for $xe^{\alpha\, x}$:
\begin{equation}\label{def:E2gen1}
 x e^{\alpha\, x}=\sum_{k\in \Z} (t_k+h)e^{\alpha \, t_k} B^h_{\alpha h,\alpha h}(x-t_k).
\end{equation}
Next, we use the fact that the convolution product of $B^h_{-h\alpha,-h\alpha}$ with the functions $f_1(x)=e^{\alpha x}$ or $f_2(x)=xe^{\alpha x} $  yields another exponential polynomial of the same type, that is
$$
B^h_{-\alpha h,h\alpha h}\ast f_1=a_0f_1,\quad \hbox{and}\quad  B^h_{-\alpha h, \alpha h}\\ast f_2=b_0f_1+b_1f_2,\quad \hbox{with}\quad a_0,b_0,b_1\in \R.
$$
Therefore, since $B^1_{\bal}=B^1_{\alpha,\alpha}\ast B^1_{-\alpha,-\alpha}$ and  $B^h_{\bal h}=\frac1 h\left(B^h_{\alpha h, \alpha h}\ast B^h_{-\alpha h, \alpha h}\right)$,  if  we convolve both side of  \eqref{def:E2gen0} and \eqref{def:E2gen1} with $B^h_{-\alpha h, -\alpha h }$, we arrive at
$$
a_0f_1=\sum_{k\in \Z} e^{\alpha\, t_k} h B^h_{\bal h}(\cdot-t_k),\quad \hbox{and}\quad b_0f_1+b_1f_2=\sum_{k\in \Z} (t_k+h)e^{\alpha \, t_k} h B^h_{\bal h}(\cdot-t_k),
$$
that are the reproduction formulas  for $B^h_{\bal h}$ of  a function in $\E_{2,\alpha}=\{e^{\alpha x}, xe^{\alpha x}\}$. Similarly we prove the reproduction of  $\E_{2,-\alpha}=\{e^{-\alpha x}, xe^{-\alpha x}\}$ and therefore the reproduction of $\E_{4,\alpha}$.\end{proof}


\section{ HP-splines reproduction of $\E_{2, -\alpha}$ and moment preservation}\label{sec:3}
Based on the HB-splines reproduction properties shown in Proposition \ref{prop:HBreproduction} in this section we show the reproduction capabilities of HP-splines, independently to the value of the smoothing parameter $\lambda$.

\begin{prop}\label{prof:HPreproduction}
Let the data points $(x_i,y_i),\ i=1,\ldots, m$, $x_1<\cdots <x_m$, be given together with the uniform knots partition $\Xi: = \{x_1:=a=\xi_1<\xi_2 \cdots  <\xi_n=b=:x_m\}$  ($n<m$) extended with  the uniform left and right extra knots 
$\xi_{\ell}=\xi_1+(\ell-1) h,\   \ell=0, -1,-2, \ \xi_{n+\ell}=\xi_n+\ell h,\ \ell=1,2,3$ where $h=(b-a)/(n-1)$. Let $\{B^\alpha_0,\ldots,B^\alpha_{n+1}\}$ be the spline basis with segments in $\E_{4,\alpha}$ consisting of the uniform HB-splines $B^\alpha_{0}:=B^h_{\bal h}(\cdot-\xi_{-2})$ and  $B^\alpha_j=B^\alpha_0(\cdot-{jh}),\ j=1,\cdots, n+1$, with $B^h_{\bal h}$ as in \eqref{def:dilatata}. Then, if the data are taken form a function $f\in \E_{2,-\alpha}$,  \ie $y_i=f(x_i),\ i=1,\cdots,m$ with $f\in \E_{2,-\alpha}$, 
 the HP-splines $s_{hp}$ defined in \eqref{def:hp} satisfies 
$$s_{hp}(x_i)=y_i,
\quad i=1,\cdots, m.$$
\end{prop}
\begin{proof}
Form Proposition \ref{prop:HBreproduction} we know that HB-splines reproduce $\E_{4,\alpha}$ meaning that there exists a sequence of coefficients $c^f_j ,\ j=0,\cdots, n+1$ satisfying 
\begin{equation}\label{def:E4gen}
\sum_{j=0}^{n+1} c^f_j  B^\alpha_j(x)=f(x), \quad x\in [a,b],\quad f\in \E_{4, \alpha}.
\end{equation}

With the notation $\hba=(c^f_0,\cdots, c^f_{n+1})$ assuming that the data are taken form a function $f\in \E_{2,-\alpha}$, it is not difficult to see that the solution of the linear system \eqref{expansion2} is exactly $\hba$, since  from \eqref{def:E4gen} we have 

\begin{equation}
y_i=\sum_{j=0}^{n+1}\ha_j B^\alpha_j(x_i),\ i=1,\cdots,m,
\end{equation}
and, for $x\in [a,b]$,
\begin{eqnarray}
 0&=&\Delta^{h,\alpha}_2 f(x)=\sum_{j=0}^{n+1} \ha_j   \Delta^{h,\alpha}_2 B^\alpha_j(x)= \nonumber \\
&=& \sum_{j=2}^{n+1}(\Delta^{h, \alpha}_2 \hat{\ba})_jB^\alpha_j(x)+\ha_0  B^\alpha_0(x)+(\ha_1-2e^{-\alpha h}\ha_0)  B^\alpha_1(x),\label{eqRep2}
\end{eqnarray}
that, in combination with the linear independence of HB-splines, shows that $(\Delta^{h, \alpha}_2 \hat{\ba})_j=0$ for $j=2,\cdots,n+1$. The latter means that the model acts like non penalized regression and that  the reproduction capabilities of the HB-splines transfer to the HP-splines.
\end{proof}

\begin{remark}It is important to remark that, even though HB-splines reproduces $\E_{4,\alpha}$, Proposition \ref{prof:HPreproduction} shows that HP-splines reproduces  $\E_{2,-\alpha}$ only.  This limitation is due to  the specific definition of the difference operator acting as shown in \eqref{eqRep2}. A model based on $\Delta^{h, -\alpha}_2$ rather than  $\Delta^{h, \alpha}_2$ would 
reproduce $\E_{2,\alpha}$.
\end{remark}

\medskip Next, we show that HP-splines preserve the two \lq exponential\rq\  moments
\begin{prop}\label{prof:HPmoments} In the notation of Proposition \eqref{prof:HPreproduction}, denoting by $\hby={\bf B}^\alpha \hba$ the vector of predicted values with elements $\hy_i=s_{hp}(x_i),\ i=1,\cdots,m$, we have 
\begin{equation}\label{eq:moments}
\sum_{i=1}^{m}e^{-\alpha x_i}\hy_i=\sum_{i=1}^{m}e^{-\alpha x_i}y_i,\quad \hbox{and} \quad  \sum_{i=1}^{m}x_ie^{-\alpha {x_i}}\hy_i=\sum_{i=1}^{m}x_ie^{-\alpha{x_i}}y_i.
\end{equation}
\end{prop}
\begin{proof}

To see \eqref{eq:moments},  we start from the two equations defining the reproduction of $\E_{2, -\alpha}$
\begin{equation}\label{boh}
\sum_{j=0}^{n+1} c_j  B^\alpha_j(x)=e^{-\alpha x}, \quad\hbox{and}\quad  \sum_{j=0}^{n+1} d_j  B^\alpha_j(x)=xe^{-\alpha x}, \quad\hbox{with}\quad  c_j ,d_j\in \R,
\end{equation}
and evaluate them at $x_i,\ i=1,\ldots,m$. Hence,  for  $\bc=(c_0,\cdots, c_{n+1})$ and $\be=(e^{-\alpha x_1},\cdots, e^{-\alpha x_m})$ we have the equivalence 
\begin{equation}\label{boh3}
\sum_{j=0}^{n+1} c_j  B^\alpha_j(x_i)=e^{-\alpha x_i},\ i=1,\cdots,m \ \  \Leftrightarrow \ \  {\bf B}^\alpha{\bf c} =\be,
\end{equation}
and, for $\bd=(d_0,\cdots, d_{n+1})$ and $\bxe=(x_1e^{-\alpha x_1},\cdots, x_me^{-\alpha x_m})$, the equivalence
\begin{equation}\label{boh4}
\sum_{j=0}^{n+1} d_j  B^\alpha_j(x_i)=x_ie^{-\alpha x_i},\ i=1,\cdots,m \ \ \ \Leftrightarrow  \ \ {\bf B}^\alpha{\bf d} =\bxe.
\end{equation}
Now, in consideration of the linear independence of the HB-splines, with the same reasoning done in Proposition \ref{prof:HPreproduction},
from 
\begin{eqnarray}
0&=&\Delta^{h, \alpha}_2 \left( e^{-\alpha x}\right)=\Delta^{h, \alpha}_2\left(\sum_{j=0}^{n+1} c_j  B^\alpha_j(x)\right)=\nonumber \\
&=&\sum_{j=2}^{n+1} (\Delta^{h, \alpha}_2 {\bf c})_j  B^\alpha_j(x)+c_0  B^\alpha_0(x)+(c_1-2e^{-\alpha h}c_0)  B^\alpha_1(x),\nonumber 
\end{eqnarray}
we can write  
$$(\Delta^{h, \alpha}_2 {\bf c})_j=0 ,\ j=2,\ldots, n+1,$$ from which we conclude 
$
\bD^{h, \alpha}_2 {\bf c}={\bf 0}.
$
Similarly,  
\begin{eqnarray}
0&=&\Delta^{h, \alpha}_2 \left( x e^{-\alpha x}\right)=\Delta^{h, \alpha}_2\left( \sum_{j=0}^{n+1} d_j  B^\alpha_j(x)\right)=\nonumber \\
&=&\sum_{j=2}^{n+1} (\Delta^{h, \alpha}_2 {\bf d})_j  B^\alpha_j(x) +d_0  B^\alpha_0(x)+(d_1-2e^{-\alpha h}d_0) 
\nonumber 
\end{eqnarray}
implies 
\begin{equation}\label{boh5}
\quad (\Delta^{h, \alpha}_2 {\bf d})_j=0 ,\ j=2,\ldots, n+1,
\end{equation}
and therefore 
$\mathbf D^{h, \alpha}_2 {\bf d}={\bf 0}.  
$
Next, for the predicted values
$\hby={\bf B}^\alpha \hba$  we have 
\begin{equation}
\label{eqrep2}
({\bf B}^\alpha)^T\left( \by-{\bf B}^\alpha \hba \right)=\lambda {\mathbf D^{h, \alpha}_2}^T{\mathbf D^{h, \alpha}_2}\hba.
\end{equation}
Multiplication of both sides of (\ref{eqrep2}) respectively by $\bf c$ and  $\bf d$ yields
$$
{\bf c}^T({\bf B}^\alpha)^T\left( \by-{\bf B}^\alpha\hba \right)=\lambda\left({\mathbf D}^{h, \alpha}_2 {\bf c}\right)^T{\mathbf D}^{h, \alpha}_2\hba,$$
and
$$
{\bf d}^T({\bf B}^\alpha)^T\left( \by-{\bf B}^\alpha\hba \right)=\lambda \left({\mathbf D^{h, \alpha}_2} {\bf d}\right)^T{\mathbf D^{h, \alpha}_2}\hba.
$$
Now, since $\mathbf D^{h, \alpha}_2 {\bf c}=0$ and $\mathbf D^{h, \alpha}_2 {\bf d}=0$  using \eqref{boh3} and \eqref{boh4}
we arrive at 

$$
\be^T\left( \by-\hby \right)={\bf 0},\quad \hbox{and}\quad \bxe^T \left( \by-\hby \right)={\bf 0},$$
which are the vector versions of \eqref{eq:moments}.
\end{proof}

\begin{remark} Note that the exponential moments \eqref{eq:moments}, reduce to the classical moments preservation  whenever $\alpha=0$ that is to 
$$
\sum_{i=1}^{m}\hy_i=\sum_{i=1}^{m}y_i,\quad \hbox{and} \quad \sum_{i=1}^{m}x_i\hy_i=\sum_{i=1}^{m}x_iy_i.$$
\end{remark} 

\smallskip We conclude the paper with some figures showing the exponential-reproduction capabilities of HP-splines. Figures \ref{Figura1}-\ref{Figura2} refers to data taken from the exponential functions $e^{-x}$ while Figure 2 to data from the function $x\,e^{-x}$. They display  the graph of the HP-spline (black \lq$-$\rq\,) approximating the data for different selections of $\alpha$ combined with different level of absolute Gaussian noise with zero mean and standard deviation $\sigma$, both specified in the figure captions. The data sites  (red \lq$*$\rq\,)  and the spline knots location (blue \lq$\diamond$\rq\,) are also given in the figures.
 The smoothing parameter is always $\lambda=1$ since not relevant to our discussion.  The exact exponential fit is evident in absence of noise (left) while it is almost attained in case of a moderate noise (right) as well in case of an uncorrect selection of the frequency. For comparison, the graph  of the P-splines approximating the data is also given (magenta \lq$-.$\rq\,)  together with the graph of the function (blue \lq$--$\rq\,).
 \begin{figure}
\begin{subfigure}{.33\textwidth}
  \centering
  \includegraphics[width=\linewidth]{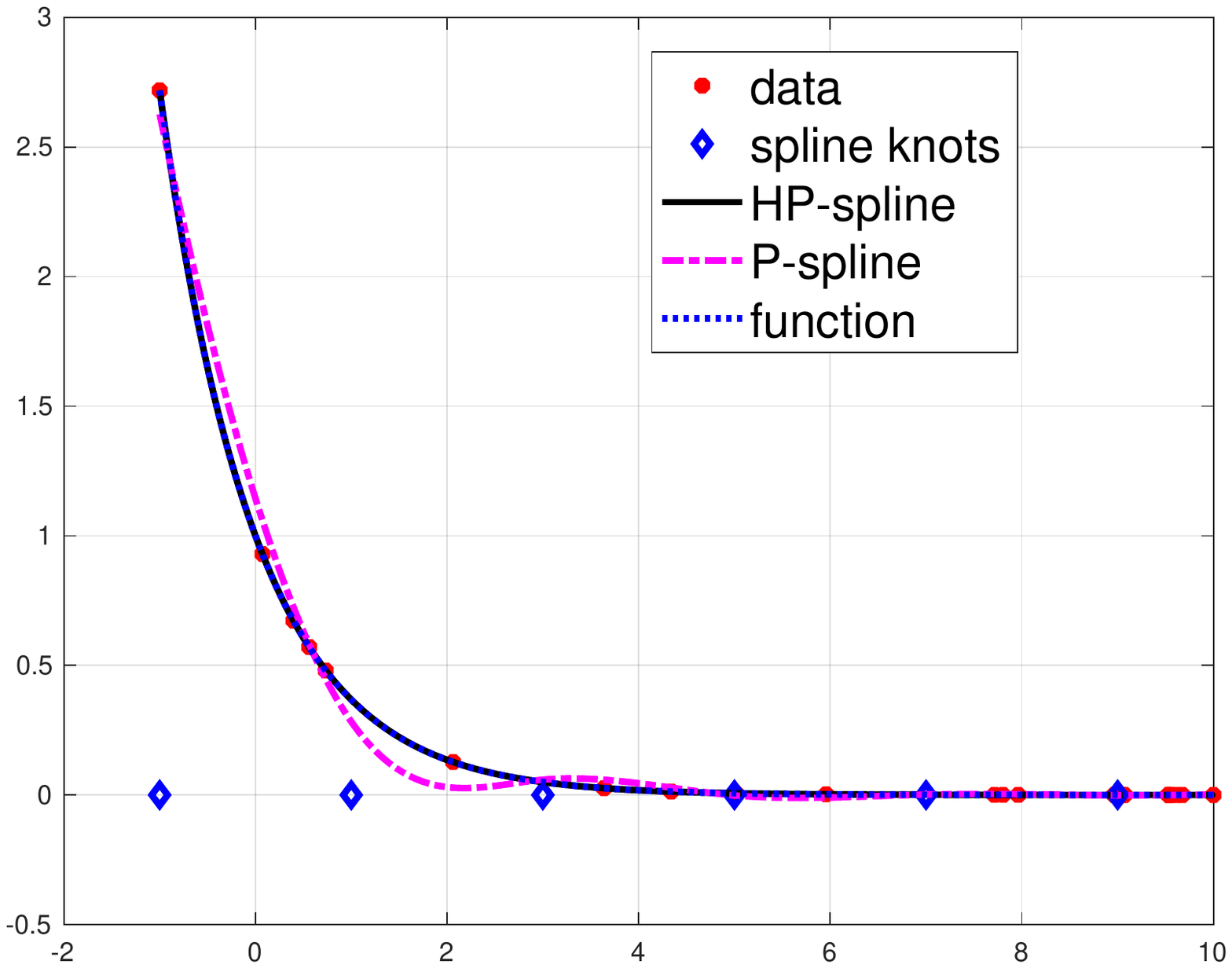}
  \label{fig:sfig1}
\end{subfigure}%
\begin{subfigure}{.33\textwidth}
  \centering
  \includegraphics[width=\linewidth]{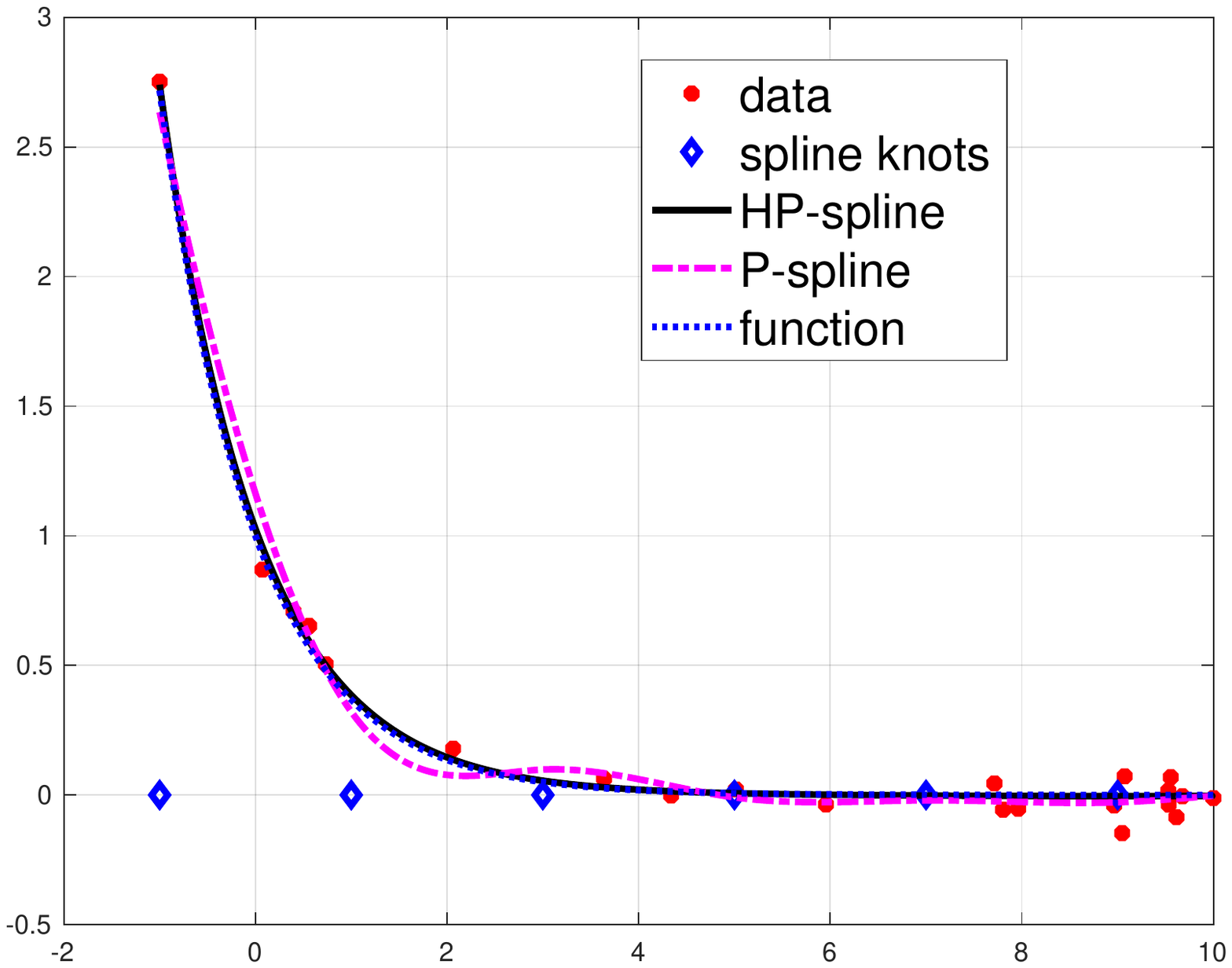}
  \label{fig:sfig2}
\end{subfigure}
\begin{subfigure}{.33\textwidth}
  \centering
  \includegraphics[width=\linewidth]{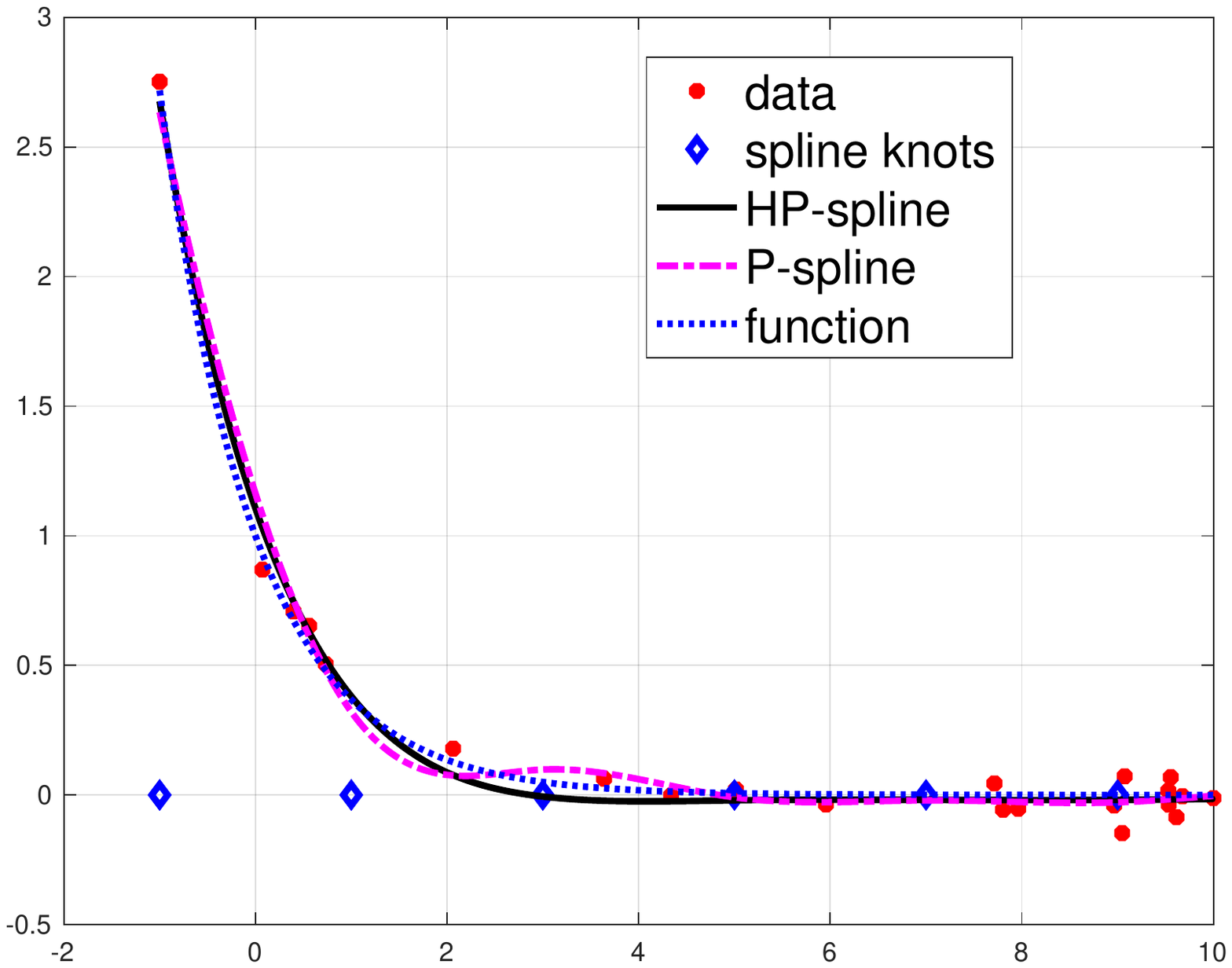}
  \label{fig:sfig3}
\end{subfigure}
 \vskip -2cm
\caption{{\footnotesize Function $e^{-x}$. From left to right values of $(\alpha,\sigma)$: $(-1, 0)$, $(-1, 0.5\cdot 10^{-2})$, $(-0.5,0.5\cdot 10^{-2})$.}}
\label{Figura1}
\end{figure} 

 \begin{figure}
\begin{subfigure}{.33\textwidth}
  \centering
  \includegraphics[width=\linewidth]{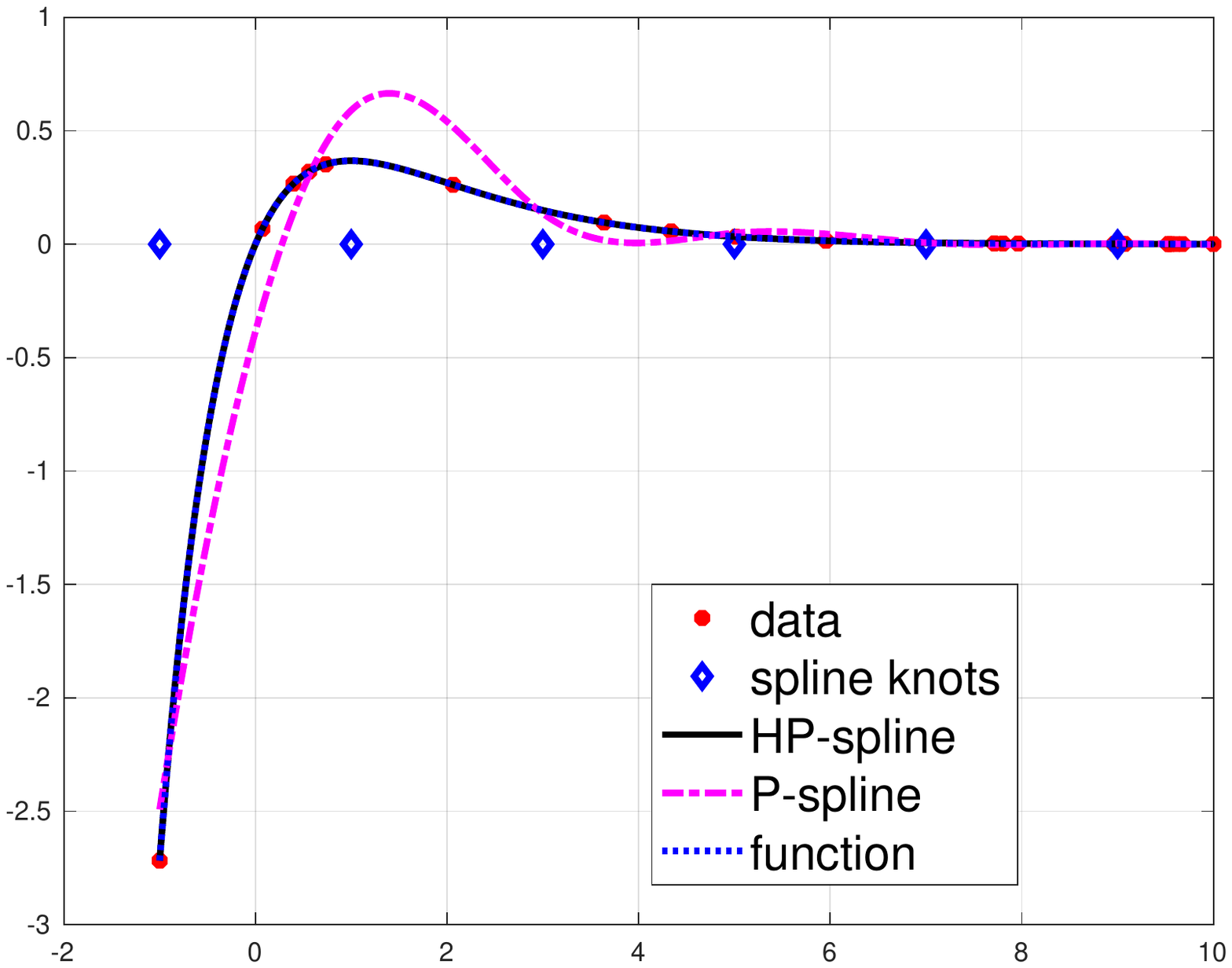}
  \label{fig:sfig21}
\end{subfigure}%
\begin{subfigure}{.33\textwidth}
  \centering
  \includegraphics[width=\linewidth]{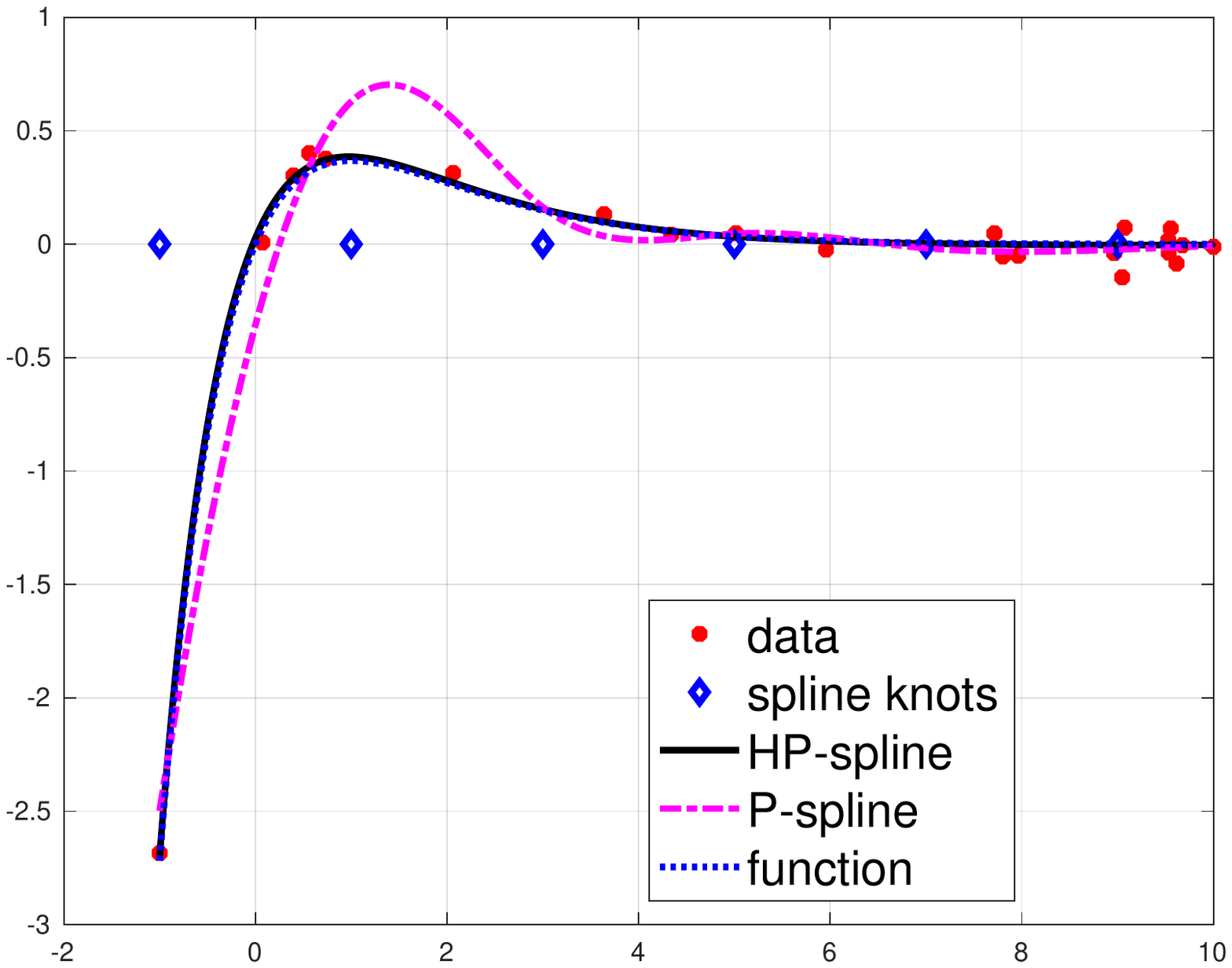}
  \label{fig:sfig22}
\end{subfigure}
\begin{subfigure}{.33\textwidth}
  \centering
  \includegraphics[width=\linewidth]{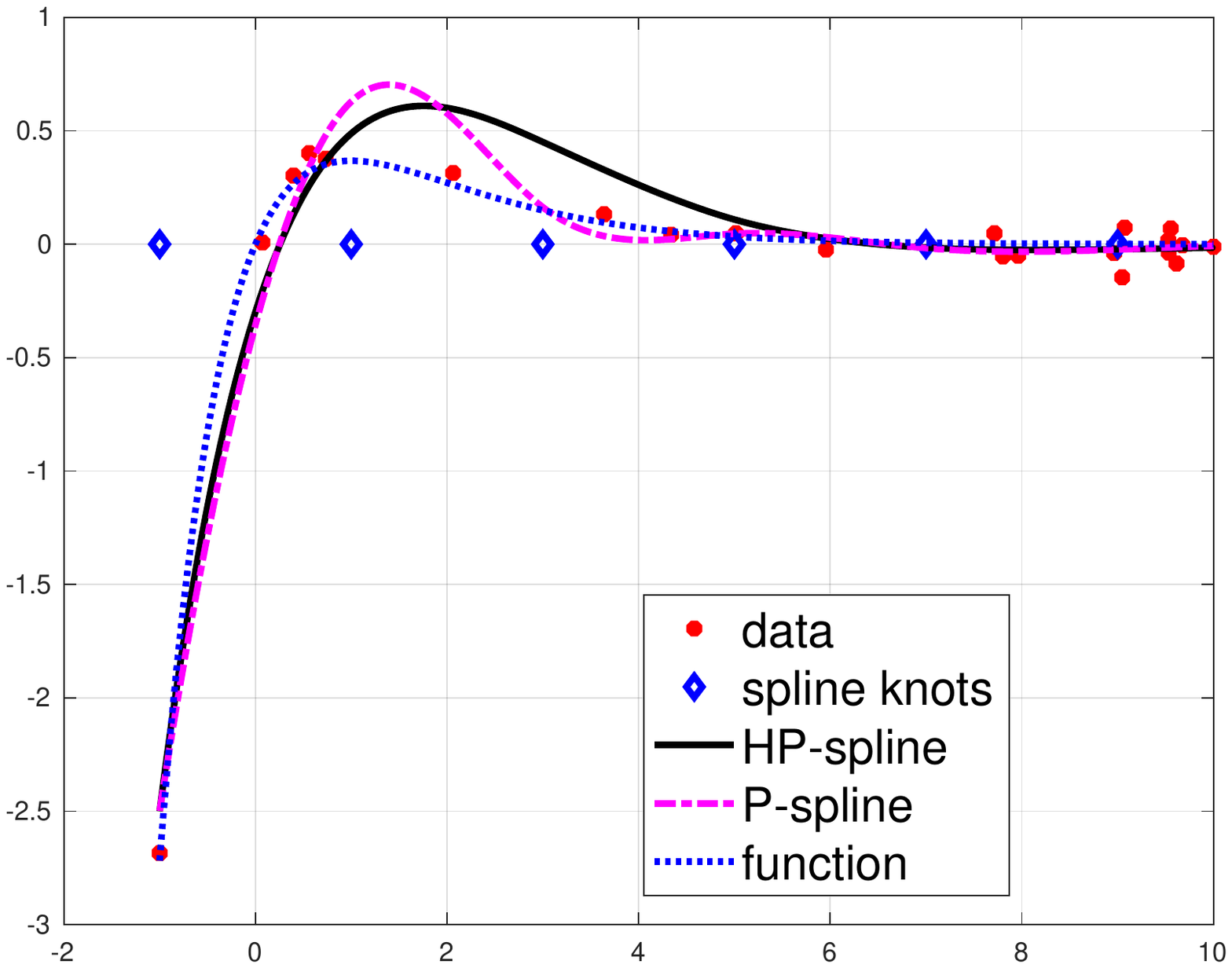}
  \label{fig:sfig23}
\end{subfigure}
 \vskip -2cm
\caption{{\footnotesize Function $xe^{-x}$. From left to right values of $(\alpha,\sigma)$: $(-1, 0)$, $(-1, 0.5\cdot 10^{-2})$, $(-0.5,0.5\cdot 10^{-2})$.}}
\label{Figura2}
\end{figure}

\section{Conclusions}\label{sec:conclusion}
This short paper enriches the study of HP-splines, penalized hyperbolic - splines with segments in the space $\E_{4,\alpha}$ consisting  in the exponential polynomials $\{e^{\alpha x},\, e^{-\alpha x},\, xe^{\alpha x},\, xe^{-\alpha x}\}$, where $\alpha$ is a real frequency. In particular, it 
investigates two important analytical properties of  HP-splines: that they exactly fit function in $\E_{2,-\alpha}$ and  that they conserve  the first and second \lq exponential\rq\ moments, independently to the value of the smoothing parameter $\lambda$.
A few numerical examples of reproduction are shown. A dynamic selection strategy of the parameter $\alpha$, that certainly deserve more attention,  is presently under investigation.

\bigskip {\bf Acknowledgement}\\
The authors are members of INdAM-GNCS, partially supporting this work. They are also member of RItA (Rete ITaliana di Approssimazione) and UMI-T.A.A. group. 

\end{document}